\documentclass[12pt]{amsart}

\usepackage[top=16mm, bottom=16mm, left=44.45mm, right=44.45mm]{geometry}

\usepackage{indentfirst}
\allowdisplaybreaks

\usepackage[pagewise]{lineno}\nolinenumbers

\usepackage{amsmath, amsfonts, amssymb, amsthm}
\usepackage{amsrefs}
\usepackage{mathrsfs}
\usepackage{enumitem} 
\usepackage{hyperref}
\usepackage{color}
\usepackage{dsfont}
\usepackage{csquotes}
\usepackage{epigraph}
\usepackage{comment}
\usepackage{mlmodern}
\usepackage[T1]{fontenc}

\usepackage{color}
\definecolor{darkblue}{rgb}{0.0,0.0,0.3}
\hypersetup{
    colorlinks=false,
    linkcolor=blue,
    urlcolor=darkblue,
    }

\newtheorem{theorem}{Theorem}[section]

\newtheorem{lemma}[theorem]{Lemma}

\theoremstyle{definition}

\newtheorem{definition}[theorem]{Definition}

\newtheorem{notation}[theorem]{Notation}

\theoremstyle{remark}
\newtheorem{remark}[theorem]{Remark}

\numberwithin{equation}{section}

\newcommand{\bR}{\mathbb{R}}
\newcommand{\bH}{\mathbb{H}}

\newcommand{\Garding}{G\r{a}rding}

\title[Curvature estimates]{Hypersurfaces of constant higher order mean curvature in hyperbolic space with prescribed asymptotic boundary at infinity}
\author{Bin Wang}
\address[]{Institute for Theoretical Sciences, Westlake University, Hangzhou, China} 
\thanks{This work was supported by the China Postdoctoral Science Foundation under Grant Number 2026M793406}
\email{wangbin@westlake.edu.cn}

\begin{document}
\begin{abstract}
In this note, we prove the existence of smooth complete admissible hypersurfaces in hyperbolic space with constant higher order mean curvature and a prescribed asymptotic boundary at infinity. The result holds for a large part of indices in the subcritical range where the concavity inequality loses its effectiveness. We supply with new arguments to reduce the proof to a semi-convex situation in which case the concavity inequality can play a role.

\end{abstract}
\maketitle
\setcounter{tocdepth}{1} 

\tableofcontents

\section{Introduction}
Fix $n \geq 3$. Let $\bH^{n+1}$ denote the hyperbolic space of dimension $n+1$ and let $\partial_{\infty} \bH^{n+1}$ denote its ideal boundary at infinity. Suppose $f: \Gamma \subseteq \bR^n \to \bR$ is a smooth symmetric function of $n$ variables satisfying some standard assumptions on an open symmetric convex cone $\Gamma \supseteq \Gamma^{+}$ with vertex at the origin that contains the positive cone \[\Gamma^{+}:=\{\kappa \in \bR^n: \kappa_i>0 \quad \forall\ 1 \leq i \leq n\}.\]
Given a closed, embedded, smooth $(n-1)$-dimensional submanifold
\begin{equation}
\Lambda \subset \partial_{\infty} \bH^{n+1} \label{asymptotic boundary}
\end{equation}
and a constant $0<\sigma<1$, we consider the asymptotic Plateau problem of finding a smooth complete hypersurface $\Sigma$ in $\bH^{n+1}$ satisfying 
\begin{equation}
f(\kappa[\Sigma])=\sigma \quad \text{and} \quad \kappa[\Sigma] \in \Gamma \quad \text{on $\Sigma$}, \label{req1}
\end{equation}
with the asymptotic boundary
\begin{equation}
\partial_{\infty} \Sigma=\Lambda \quad \text{at infinity}. \label{req2} 
\end{equation} Here $\kappa[\Sigma]=(\kappa_1,\ldots,\kappa_n)$ denotes the hyperbolic principal curvatures of $\Sigma$. 

When $\Gamma=\Gamma^{+}$, admissible hypersurfaces are locally strictly convex. For a broad class of curvature functions $f$, the existence theory in this setting was developed by Guan, Spruck, Szapiel and Xiao \cite{JGA, SGAR, JDG} and their result in \cite{JDG} is essentially optimal. Alternative proofs may be found in \cite{Xiao-A, Smith, Hong-Li-Zhang}.

Guan and Spruck \cite{JEMS} subsequently treated the problem in a general cone $\Gamma$ and obtained the existence of solutions for $\sigma \in (\sigma_0,1)$, where $\sigma_0 \approx 0.37$. Xiao \cite{Xiao-B} refined the computations and lowered the threshold to $\sigma_0 \approx 0.14$. Further refinements of the same computations do not appear likely to reach the full interval $\sigma \in (0,1)$.

In this note, we aim to investigate the existence of solutions for all $\sigma \in (0,1)$ in the general cone $\Gamma$. 
Continuing our previous investigations \cite{Bin-MRL, Bin-Adv, Bin-PAMS} on the subject, we approach the problem by a different strategy. That is, instead of a general symmetric function $f(\kappa)$, we may consider two important families of curvature functions, which are the $k$-th order mean curvatures
\[f(\kappa)=H_{k}^{1/k}(\kappa)\] and their quotients
\[f(\kappa)=\left(\frac{H_k(\kappa)}{H_l(\kappa)}\right)^{\frac{1}{k-l}}, \quad \text{where} \quad 1 \leq l<k \leq n,\] and
\[H_k(\kappa)=\binom{n}{k}^{-1}\sigma_k(\kappa):=\binom{n}{k}^{-1}\sum_{1 \leq i_1<i_2<\cdots<i_k \leq n} \kappa_{i_1}\cdots \kappa_{i_k}\] is the normalized $k$-th elementary symmetric polynomial of principal curvatures. The natural ellipticity cone for both families is the $k$-th \Garding\ cone,
\[\Gamma=\Gamma_k:=\{\kappa \in \bR^n: H_j(\kappa)>0 \quad \forall\ 1 \leq j \leq k\}.\]

These curvature functions are the prototypical examples of fully nonlinear elliptic Weingarten curvature equations and they have been extensively studied since the seminal work \cite{CNS-3, CNS-4, CNS-5} of Caffarelli, Nirenberg, and Spruck. They are important because of their natural appearance in geometric problems. In particular, they include the following notable examples.
\[\begin{alignedat}{2}
H_1(\kappa)&=\frac{1}{n}\sum_{i=1}^{n} \kappa_i &\quad& \text{is the mean curvature},\\
H_2(\kappa)&=\frac{2}{n(n-1)}\sum_{i<j}\kappa_i\kappa_j &\quad & \text{is the scalar curvature},\\
H_n(\kappa)&=\prod_{i=1}^{n}\kappa_i &\quad &  \text{is the Gauss curvature}, \\
\left(\frac{H_n}{H_{n-1}}\right)(\kappa)&=n\left(\sum_{i=1}^{n} \frac{1}{\kappa_i}\right)^{-1} &\quad &\text{is the harmonic curvature}.
\end{alignedat}\] For other values of $k$ and $l$, these curvature functions  also arise naturally in many geometric problems. 

In the pioneering work \cite{Anderson-1, Anderson-2}, Michael.~T.~Anderson established various existence theorems for complete minimal varieties in $\bH^{n+1}$ with prescribed asymptotic boundary. In particular, Anderson \cite[Theorem 10]{Anderson-1} solved the Dirichlet problem for complete minimal graphs in $\bH^{n+1}$ over a mean-convex domain, which corresponds to our problem \eqref{the Dirichlet problem} with
\[f(\kappa)=H_1(\kappa)=0.\] The boundary regularity at infinity and the extension to the constant mean curvature equation 
\[f(\kappa)=H_{1}(\kappa)=\sigma \in (0,1)\]
were further developed by Hardt-Lin \cite{Hardt-Lin}, Lin \cite{Lin}, Tonegawa \cite{Tonegawa}, Nelli-Spruck \cite{Nelli-Spruck}, and Guan-Spruck \cite{AJM}. 

We are concerned with the natural fully nonlinear generalization,
\[f(\kappa)=H_{k}^{1/k}(\kappa)=\sigma \in (0,1), \quad \text{where $2 \leq k \leq n$}.\] For the case of constant Gauss curvature $(k=n)$, the problem was initiated by Labourie \cite{Labourie} in dimension $n=2$ and was later resolved by Rosenberg and Spruck \cite{Rosenberg-Spruck} in all higher dimensions $n \geq 3$. The scalar curvature case $(k=2)$ in dimension $n=3$ follows from a result of Lu \cite{Lu-PAMS} and was recently resolved by Wang \cite{Bin-Adv} in higher dimensions $n \geq 4$.

For the harmonic curvature $f=H_n/H_{n-1}$, the natural ellipticity cone is the positive cone $\Gamma_n=\Gamma^{+}$ and hence the existence result for all $\sigma \in (0,1)$ is covered by Guan-Spruck-Szapiel-Xiao's treatment \cite{JGA, SGAR, JDG} for locally strictly convex solutions. For the more general quotient $H_k/H_{k-1}$ in the $k$-th cone $\Gamma_k$, the existence result was obtained by Wang in \cite{Bin-MRL}. For the $k$-curvatures $H_{k}^{1/k}$ in $\Gamma_k$, where $3 \leq k \leq n-1$, Lu \cite{Lu-PAMS} settled the case $k=n-1, n \geq 3$, and Wang followed the same approach to resolve the case $k=n-2, n\geq 5$ in \cite{Bin-PAMS}. Thus, among the $k$-curvatures, the remaining cases are $3 \leq k \leq n-3$.

Our main result settles the problem for a large part of this range. In other words, we shall prove that any mean-convex $\Lambda$ in \eqref{asymptotic boundary} is the asymptotic boundary of a complete embedded hypersurface of constant $k$-curvature in $\bH^{n+1}$.
\begin{theorem} \label{the theorem}
Let $n \geq 3$ and $2 \leq k \leq n$ satisfy 
\[\text{either} \quad \frac{3-\sqrt{5}}{2}<\frac{2k}{n} \leq 1 \quad \text{or} \quad \frac{2k}{n}>1.\]
Suppose  
\[\Lambda \subset \partial_{\infty} \bH^{n+1}\] is a closed, embedded, smooth codimension one submanifold that can be viewed as the boundary $\partial \Omega$ of a smooth bounded domain $\Omega \subseteq \bR^n$ in the upper half-space model of $\bH^{n+1}$. Let $0<\sigma<1$ also be a given constant. If $\partial \Omega$ has non-negative mean curvature with respect to the inward unit normal, computed in the Euclidean metric, then there exists a smooth complete hypersurface $\Sigma$ in $\bH^{n+1}$ satisfying 
\[H_{k}^{1/k}(\kappa[\Sigma])=\sigma \quad \text{and} \quad \kappa[\Sigma] \in \Gamma_k \quad \text{on $\Sigma$} \]
with the asymptotic boundary
\[\partial_{\infty} \Sigma=\Lambda \quad \text{at infinity}.\] Moreover, $\Sigma$ is the graph of some $u \in C^{\infty}(\Omega) \cap C^1(\overline{\Omega})$ over $\Omega$.
\end{theorem}

\begin{remark}
Since the open cases for the $k$-curvatures are $3 \leq k \leq n-3$, we would like the result to hold for, e.g.,
\[\frac{2k}{n} \geq \frac{6}{n}.\] In case $n \geq 16$ is large, if one could improve the lower bound
\[\text{from} \quad 0.382 \approx \frac{3-\sqrt{5}}{2} \quad \text{to} \quad \frac{6}{n},\] then the problem would be fully settled for all the $k$-curvatures. The only ingredient of our proof that does not hold for the gap
\begin{equation}
\frac{6}{n} \leq \frac{2k}{n} \leq \frac{3-\sqrt{5}}{2}\label{the gap}
\end{equation} is Lemma \ref{the claim 2}; otherwise our approach is completely general.
\end{remark}

\begin{remark}
It might be possible to improve the threshold
\[\frac{3-\sqrt{5}}{2}\] by refining our computations. We may perhaps return to this subject at a later time.
\end{remark}

For the proof, we follow Guan and Spruck \cite{JEMS} to work in the half-space model of $\bH^{n+1}$ and seek the desired hypersurface $\Sigma$ as a vertical graph of some positive smooth function $u$ over a bounded domain $\Omega \subseteq \bR^n$ with mean-convex boundary. The conditions \eqref{req1}-\eqref{req2} then reduce to the following Dirichlet problem for a fully nonlinear elliptic equation,
\begin{equation} \label{the Dirichlet problem}
\begin{alignedat}{2}
G(D^2u,Du,u)&=\sigma, &\quad u>0, \quad & \text{in $\Omega$},\\
u&=0 &\quad & \text{on $\partial \Omega$}.\\
\end{alignedat}
\end{equation} 
Since the equation degenerates at the boundary where $u=0$, we first solve, for small $\varepsilon>0$, the approximate problem
\begin{equation} \label{the approximate problem}
\begin{alignedat}{2}
G(D^2u,Du,u)&=\sigma, &\quad u>0, \quad & \text{in $\Omega$},\\
u&=\varepsilon &\quad & \text{on $\partial \Omega$},\\
\end{alignedat}
\end{equation} and then take the limit $\varepsilon \to 0$ to obtain a solution $u \in C^{\infty}(\Omega) \cap C^1(\overline{\Omega})$ of \eqref{the Dirichlet problem}, whose graph is complete and has the asymptotic boundary $\Lambda=\partial \Omega$.

We shall apply the standard continuity method to solve the approximate problem \eqref{the approximate problem}, for which we need to establish a priori estimates for admissible solutions up to the second order. In \cite{JEMS}, Guan and Spruck have perfectly established all the required estimates, especially the very delicate boundary $C^2$ estimate,
\begin{equation}
\max_{\partial \Omega} u|D^2u| \leq C,\label{boundary C2}
\end{equation} where $C>0$ is independent of $\varepsilon$. Combining this with the global-to-boundary $C^2$ estimate,
\[u|D^2u| \leq C \left(1+\max_{\partial \Omega} u|D^2u|\right)\frac{u^2}{\varepsilon^2} \quad \text{in $\Omega$},\]
that they obtained in the joint work \cite[Theorem 5.2]{JGA} with Szapiel, they concluded that
\begin{equation}
\max_{\overline{\Omega}} u|D^2u| \leq \frac{C}{\varepsilon^2}, \label{global C2}
\end{equation} where $C>0$ is independent of $\varepsilon$. This global estimate was sufficient for them to solve the approximate problem \eqref{the approximate problem}, but it forbids the passage to the limit as $\varepsilon \to 0$.

The additional ingredient needed for the limiting process is a new control on the hyperbolic principal curvatures
\begin{equation}
\max_{\Omega} \kappa_{\max} \leq C \left(1+ \max_{\partial \Omega} \kappa_{\max}\right), \label{global C2 curvature}
\end{equation} where $C>0$ is independent of $\varepsilon$. This is the one and only place where Guan and Spruck \cite[Theorem 6.1]{JEMS} had to restrict the range of $\sigma$ to $(\sigma_0,1)$. Hence, in order to prove Theorem \ref{the theorem}, the task is reduced to proving the curvature estimate \eqref{global C2 curvature} for all $\sigma \in (0,1)$.

When deriving \eqref{global C2 curvature} by the standard maximum-principle argument, there will be some additional negative terms that have no counterparts in the Euclidean space. This prevents us from a direct application of the novel techniques due to Guan-Li-Li \cite{Guan-Li-Li}, Guan-Ren-Wang \cite{Guan-Ren-Wang}, and Sheng-Urbas-Wang \cite{Sheng-Urbas-Wang} for deriving curvature estimates for the $k$-curvature equations in $\bR^{n+1}$. The recent techniques developed by Lu \cite{Lu-CVPDE} and Yang \cite{Yang} for deriving curvature estimates in $\bH^{n+1}$ cannot be adapted to our setting as well. We discuss the precise obstruction in Section \ref{the difficulties} below.

New elements were introduced by Lu in a separate work \cite{Lu-PAMS}. Firstly, Lu applied the maximum-principle argument to an auxiliary function that is different from the one used by Guan and Spruck \cite{SGAR}. The resulting positive terms can now absorb the most troublesome negative term. However, there arises in correspondence another potentially large negative term. To resolve this issue, Lu used a powerful concavity inequality due to Ren and Wang \cite{Ren-Wang-1} which was derived for the $H_{n-1}^{1/(n-1)}$ operator. The inequality enabled Lu to eliminate the newly arisen negative term, and hence proved the existence result for $k=n-1, n \geq 3$. Subsequently, Wang \cite{Bin-PAMS} followed the same argument to get the existence for $k=n-2, n \geq 5$, by making use of the analogous inequality of Ren-Wang \cite{Ren-Wang-2} that was derived for the $H_{n-2}^{1/(n-2)}$ operator.

For the intermediate cases $3 \leq k \leq n-3$, Wang \cite{Bin-MRL} was able to prove the curvature estimate \eqref{global C2 curvature} in the smaller cone $\Gamma_{k+1} \subseteq \Gamma_{k}$. Hong and Zhang \cite{Hong-Zhang} were able to derive \eqref{global C2 curvature} for $\Gamma_k$-admissible solutions that additionally satisfy a uniform curvature lower bound,
\begin{equation}
\exists\ A>0: \quad \kappa_i \geq -A \quad \text{for all} \quad 1 \leq i \leq n \quad \text{on $\Sigma$}. \label{lower bound assumption}
\end{equation}
This lower bound assumption was used by them to derive a concavity inequality for the operator $H_{k}^{1/k}$ that improves an earlier inequality of Lu \cite{Lu-CVPDE}. Their estimate was an important step forward, but it does not by itself yield the existence theorem, because the bound \eqref{lower bound assumption} is not available a priori. 

Recently, a new concavity inequality for the $H_{k}^{1/k}$ operator has been verified by Yan \cite{Yan}, which holds without the lower bound assumption \eqref{lower bound assumption}; the inequality was conjectured to be true by Ren and Wang in \cite{Ren-Wang-2}. It seems now that we should be able to get Theorem \ref{the theorem} for all $k$ by combing Lu's argument \cite{Lu-PAMS} with Yan's inequality \cite{Yan}. However, we clarify that this is not yet the case. When $2k>n$, Yan's inequality is as powerful as the ones due to Ren-Wang \cite{Ren-Wang-1, Ren-Wang-2}, so in this case, we would be able to obtain the estimate \eqref{global C2 curvature} for all $\sigma \in (0,1)$, by following Lu's derivation in \cite{Lu-PAMS}. In contrast, when $2k = n$, even though Yan's inequality is still valid, the lower bound of a particularly important term whose coefficient depends on $2k-n$, degenerates to $0$ and Lu's argument cannot be fully invoked. The situation is apparently even worse when $2k<n$ in which case the concavity inequality of Yan also starts to lose its effectiveness. 

The new contribution of Theorem \ref{the theorem} is the treatment for a large part of the range $2k \leq n$. We adapt a reduction technique from Wang's work \cite{Bin-Adv} on the case $k=2, n \geq 4$ to make the lower bound condition \eqref{lower bound assumption} available. This provides a positive lower bound for the particular term that depends on the value of $2k-n$ and permits the use of Hong-Zhang's concavity inequality \cite{Hong-Zhang, Zhang}. Combining these ingredients with Lu's argument in \cite{Lu-PAMS}, we are able to obtain the curvature estimate \eqref{global C2 curvature} when
\[\frac{3-\sqrt{5}}{2}<\frac{2k}{n} \leq 1.\]

\begin{remark}
In \cite{Yan}, Yan applied the concavity inequality to obtain $C^2$ estimates for the $k$-curvature equations in $\bR^{n+1}$ when $2k \geq n$. In that space, the work of Guan-Ren-Wang \cite{Guan-Ren-Wang} and subsequent developments \cite{Lu-CVPDE, Yang, Ren-Wang-1, Ren-Wang-2, Li-Ren-Wang} would make the proof of $C^2$ estimates essentially standard once an appropriate concavity inequality is available; the borderline case $2k=n$ causes no additional difficulty.
However, in our setting, by contrast, the negative sectional curvature of $\bH^{n+1}$ produces several more negative terms, and the equation also degenerates at the asymptotic boundary. These features are responsible for the substantial difficulties when $2k \leq n$; see Remark \ref{the degeneracy} below.
\end{remark}

The rest of this note is organized as follows. In Section \ref{Preliminaries}, we fix the notation, review the upper half-space model of $\bH^{n+1}$, and recall some properties of the elementary symmetric polynomials
\[\sigma_k(\kappa)=\sum_{1 \leq i_1<\cdots<i_k \leq n} \kappa_{i_1}\cdots \kappa_{i_k}.\] In Section 
\ref{2k<=n proof}, we derive the required curvature estimate \eqref{global C2 curvature} when
\[\frac{3-\sqrt{5}}{2}<\frac{2k}{n} \leq 1.\] In Section \ref{2k>n proof}, we derive the curvature estimate when
\[\frac{2k}{n}>1.\] Once we have the curvature estimates, Theorem \ref{the theorem} would follow as a consequence of the general framework of Guan and Spruck in \cite{JEMS}.

\begin{remark}
After this work was posted on arXiv, we found a preprint \cite{Mei-Yan} of Mei and Yan that was posted on the same day. In that preprint, they were able to get our main result for all $2k \leq n$, while we need to impose 
\begin{equation}
\frac{2k}{n}>\frac{3-\sqrt{5}}{2} \approx 0.382 \label{the restriction}
\end{equation} in Theorem \ref{the theorem}.

Being inspired by a keen observation (see \eqref{the observation} below) in their argument, we found that one key step of our proof can be slightly revised so that Lemma \ref{the claim 2} can hold for the gap range \eqref{the gap} as well. In other words, we can remove the restriction \eqref{the restriction} from Theorem \ref{the theorem} and the problem \eqref{req1}-\eqref{req2} has now been fully settled for all the $k$-curvatures.

We shall keep Theorem \ref{the theorem} in the original form and also not alter the proof in Section \ref{2k<=n proof}. However, we shall include an appendix to show the extension of Lemma \ref{the claim 2} to all $2k \leq n$, as our proof is different from Mei-Yan in character and may stay helpful to some readers.

Mei and Yan were the first to obtain Theorem \ref{the theorem} for all the intermediate cases $3 \leq k \leq n-3$.
\end{remark}

\section{Preliminaries} \label{Preliminaries}

\subsection{The upper half-space model for $\bH^{n+1}$}
\indent

Following Guan-Spruck \cite{JEMS}, we use the upper half-space model for the hyperbolic space, 
\[\bH^{n+1}=\{(x,x_{n+1}) \in \bR^{n+1}: x_{n+1}>0\},\] that is equipped with the metric
\[ds^2=\frac{1}{x_{n+1}^2}\sum_{i=1}^{n+1} dx_{i}^2.\] Thus, the ideal boundary $\partial_{\infty} \bH^{n+1}$ is naturally identified with $\bR^n=\bR^n \times \{0\} \subseteq \bR^{n+1}$ and \eqref{req2} may be understood in the Euclidean sense. In this model, we view $\Lambda=\partial \Omega$ as the boundary of a smooth bounded domain $\Omega$ in $\bR^n$.

\begin{definition}
We say the smooth bounded domain $\Omega$ is mean-convex, if the mean curvature of the boundary $\partial \Omega$ with respect to the inward unit normal, computed in the Euclidean metric, is non-negative.
\end{definition}

Suppose $\Sigma=\{(x,u(x)): x \in \Omega\}$ is the vertical graph of a function $u \in C^2(\Omega)$ over $\Omega$, which is oriented by the upward Euclidean unit normal vector field
\[\nu=\left(-\frac{Du}{w}, \frac{1}{w}\right), \quad \text{where $w=\sqrt{1+|Du|^2}$}.\] The component
\[\nu^{n+1}=\frac{1}{\sqrt{1+|Du|^2}}\] will prove to be useful in subsequent sections. The symbols $g$ and $\nabla$ will be used to denote the induced hyperbolic metric and the corresponding Levi-Civita connection on $\Sigma$. Since $\Sigma$ is also a submanifold of $\bR^{n+1}$, we will add a ``tilde'' over geometric quantities that are with respect to the Euclidean metric, $\tilde{g}$ and $\tilde{\nabla}$. In particular, some formulas \cite[Lemma 4.1]{Sui-CVPDE} that we will use are the following,
\begin{gather}\label{geometric formula}
\begin{split}
\tilde{g}^{jk}&=\frac{\delta_{jk}}{u^2}, \quad h_{ij}=\frac{1}{u}\tilde{h}_{ij}+\frac{\nu^{n+1}}{u^2}\tilde{g}_{ij},\\
\nabla_i \nu^{n+1}&=-\tilde{h}_{ij}\tilde{g}^{jk}\nabla_k u, \quad \sum_{i=1}^{n} \frac{(\nabla_iu)^2}{u^2}=|\tilde{\nabla}u|^2=1-(\nu^{n+1})^2 \leq 1,
\end{split}
\end{gather} in a local orthonormal frame on $\Sigma$.

\subsection{Properties of the $\sigma_k$ operator}
\indent

Recall the definition of the $k$-th elementary symmetric polynomial,
\[\sigma_k(\kappa_1,\ldots,\kappa_n):=\sum_{1 \leq i_1<\cdots<i_k \leq n} \kappa_{i_1}\cdots \kappa_{i_k},\] and we adopt the convention that $\sigma_0:=1$ and $\sigma_k:=0$ for $k>n$. The natural ellipticity cone is $k$-th \Garding\ cone, which is an open symmetric convex cone defined by 
\[\Gamma_k:=\{\kappa \in \bR^n: \sigma_j(\kappa)>0 \quad \forall\ 1 \leq j \leq k\}.\]

\begin{notation}
Observe that
\[\frac{\partial}{\partial \kappa_i}\sigma_k(\kappa)=\sigma_{k-1}(\kappa)\bigg|_{\kappa_i=0}=\sigma_{k-1}(\kappa_1,\ldots,\kappa_{i-1},0,\kappa_{i+1},\ldots,\kappa_n).\] Therefore, we may use $\sigma_{k-1}(\kappa|i)$ to denote the first order derivatives of $\sigma_k(\kappa)$. The notation $\sigma_{k-2}(\kappa|ij)$ is defined in the same fashion for the second order derivatives. 
\end{notation}

The following are some commonly used properties of the $\sigma_k$ operator and we may state them without proofs.
\begin{lemma} \label{sigma_k properties 1}
For all $1 \leq k \leq n$ and $\kappa \in \bR^n$, we have
\begin{align*}
&\ \sigma_k(\kappa)=\kappa_i\sigma_{k-1}(\kappa|i)+\sigma_{k}(\kappa|i) \quad \text{for all $1 \leq i \leq n$},\\
&\ \sum_{i=1}^{n}\kappa_i\sigma_{k-1}(\kappa|i)=k\sigma_k(\kappa), \\
&\ \sum_{i=1}^{n} \sigma_{k-1}(\kappa|i)=(n-k+1)\sigma_{k-1}(\kappa).
\end{align*}
\end{lemma}

\begin{lemma} \label{sigma_k properties 2}
Let $\kappa \in \Gamma_k$ be ordered as 
\[\kappa_1 \geq \cdots \geq \kappa_n.\] Then $\kappa_k>0$ and 
\[0<\frac{\partial\sigma_k}{\partial \kappa_1} (\kappa) \leq \cdots \leq \frac{\partial \sigma_k}{\partial \kappa_n}(\kappa).\] If $\kappa_i \leq 0$, then
\[\kappa_i>-\frac{n-k}{k} \kappa_1\] and 
\[\sigma_{k-1}(\kappa|i) \geq C(n,k) \sigma_{k-1}(\kappa).\]
\end{lemma}

\begin{lemma} \label{NM}
Let $n \geq k>l \geq 0$ and $n \geq r>s \geq 0$. Suppose $\kappa \in \Gamma_k$. If $k \geq r$ and $l \geq s$, then
\[\left[\frac{\binom{n}{k}^{-1}\sigma_k(\kappa)}{\binom{n}{l}^{-1}\sigma_l(\kappa)}\right]^{\frac{1}{k-l}} \leq \left[\frac{\binom{n}{r}^{-1} \sigma_r(\kappa)}{\binom{n}{s}^{-1}\sigma_s(\kappa)}\right]^{\frac{1}{r-s}}.\] 
\end{lemma}
\begin{notation}
The $\sigma_k$ operator can be regarded as a function on the space of $n \times n$ symmetric matrices, by writing $\sigma_k(A)$ to mean $\sigma_k(\lambda(A))$. In particular, $\sigma_k(A)$ is the sum of the $k \times k$ principal minors of $A$.
With this abuse of notation, we may also write $\sigma_{k}^{ij}$ and $\sigma_{k}^{ij,kl}$ to mean the first and the second order derivatives in the sense of \eqref{derivatives of F}. When the matrix is diagonal, the notations $\sigma_{k}^{ii}$ and $\sigma_{k}^{ii,jj}$ can be used to denote the derivatives in the sense of \eqref{derivatives of f} without causing confusion. 
\end{notation}

When $A=a_{ij}=\kappa_i\delta_{ij}$ is diagonalized with eigenvalues $\kappa_1 \geq \cdots \geq \kappa_n$, we have
\[
\sigma_k^{pq}(A)
=
\frac{\partial \sigma_k}{\partial \kappa_p}(\kappa)\delta_{pq}
=
\sigma_{k-1}(\kappa\mid p)\delta_{pq},
\]

\begin{equation}
\sigma_k^{pq,rs}(A)
=
\begin{cases}
\displaystyle
\sigma_{k-2}(\kappa\mid pr),
&
p=q,\ r=s,\ p\neq r,
\\[10pt]
\displaystyle
-\sigma_{k-2}(\kappa\mid pq),
&
p=s,\ q=r,\ p\neq q,
\\[10pt]
0,
&
\text{otherwise}.
\end{cases} \label{2nd derivatives for sigma k}
\end{equation} Moreover, by Lemma \ref{sigma_k properties 1} and Lemma \ref{sigma_k properties 2}, one can verify that
\begin{equation}
\sigma_{k}^{pp,qq}=\frac{\sigma_{k}^{pp}-\sigma_{k}^{qq}}{\kappa_q-\kappa_p} \geq 0 \quad \text{if $\kappa_p \neq \kappa_q$}.\label{2nd derivative for sigma k 2}
\end{equation}

\begin{lemma} \label{the concavity inequality 1}
Let $n \geq 3$ and $2 \leq k \leq n-1$. Given two positive constants $0 < \psi_0 \leq \psi_1 < \infty$, assume the set
\[\widetilde{\Gamma}_k:=\{\kappa \in \Gamma_k: \psi_0 \leq \sigma_k(\kappa) \leq \psi_1\}\] is non-empty. In addition, assume the vector $\kappa \in \widetilde{\Gamma}_k$ is ordered with multiplicity $m \geq 2$ in the sense that
\[\kappa_1 = \cdots = \kappa_m > \kappa_{m+1} \geq \cdots \geq \kappa_n.\] 
For 
\begin{equation}
0<\gamma_0<\min \left\{\frac{2k}{n}, 1+ \frac{2k-n}{2k^2+n}\right\},\label{gamma}
\end{equation} there exists some $\eta_{*}=\eta_{*}(n,k,\gamma_0, \psi_0,\psi_1)>0$ such that if 
\[\kappa_1 \geq \eta_{*}>0,\] then
\begin{gather} \label{concavity 1}
\begin{split}
& -\sum_{p \neq q} \frac{\sigma_{k}^{pp,qq} \xi_p\xi_q}{\kappa_1} - \gamma_0 \frac{\sigma_{k}^{11} \xi_{1}^2}{\kappa_{1}^2}+2\sum_{i>m} \frac{\sigma_{k}^{ii}\xi_{i}^2}{\kappa_1(\kappa_1-\kappa_i)}\\
\geq &\ - \frac{2}{\kappa_1 \sigma_k} \left(\sum_{i=1}^{n} \sigma_{k}^{ii}\xi_{i}\right)^2
\end{split}
\end{gather} for any $\xi \in \bR^n$ with
\[\xi_2 = \cdots = \xi_m = 0,\] where the $\sigma_k$ and its derivatives are evaluated at the given vector $\kappa \in \widetilde{\Gamma}_k$. If $m=1$, then the inequality \eqref{concavity 1} holds for all $\xi \in \bR^n$.
\end{lemma}
\begin{proof}
See \cite[Theorem 1.1 and Corollary 2.3]{Yan}.
\end{proof}

\begin{lemma} \label{the concavity inequality 2}
Let $n \geq 3$ and $2 \leq k \leq n-1$. Given two positive constants $0 < \psi_0 \leq \psi_1 < \infty$, assume the set
\[\widetilde{\Gamma}_k:=\{\kappa \in \Gamma_k: \psi_0 \leq \sigma_k(\kappa) \leq \psi_1\}\] is non-empty. In addition, assume the vector $\kappa \in \widetilde{\Gamma}_k$ is ordered with multiplicity $m \geq 2$ in the sense that
\[\kappa_1 = \cdots = \kappa_m > \kappa_{m+1} \geq \cdots \geq \kappa_n,\] and
\[\kappa_n \geq -A\] for some constant $A>0$.
For a small enough $\delta_0 \in (0,1)$, there exist positive constants $\eta_{*}$ and $K$ whose values depend only on $n$, $k$, $A$, $\delta_0$, $\psi_0$ and $\psi_1$, such that if 
\[\kappa_1 \geq \eta_{*}>0,\] then
\begin{gather} \label{concavity 2}
\begin{split}
& -\sum_{p \neq q} \frac{\sigma_{k}^{pp,qq} \xi_p\xi_q}{\kappa_1} -(1+\delta_0) \frac{\sigma_{k}^{11} \xi_{1}^2}{\kappa_{1}^2}+2\sum_{i>m} \frac{\sigma_{k}^{ii}\xi_{i}^2}{\kappa_1(\kappa_1-\kappa_i)}\\
\geq &\ - \frac{K}{\kappa_1 \sigma_k} \left(\sum_{i=1}^{n} \sigma_{k}^{ii}\xi_{i}\right)^2
\end{split}
\end{gather} for any $\xi \in \bR^n$ with
\[\xi_2 = \cdots = \xi_m = 0,\] where the $\sigma_k$ and its derivatives are evaluated at the given vector $\kappa \in \widetilde{\Gamma}_k$. If $m=1$, then the inequality \eqref{concavity 2} holds for all $\xi \in \bR^n$.
\end{lemma}
\begin{proof}
When $m=1$, the inequality follows from \cite[Lemma 1.1]{Zhang}. If $m \geq 2$, apply the result for $m=1$ to
\[\kappa^{\varepsilon}:=(\kappa_1,\kappa_2-\varepsilon, \ldots,\kappa_m-\varepsilon,\kappa_{m+1},\ldots,\kappa_n)\] and any vector $\xi \in \bR^n$ with
\[\xi_2=\cdots=\xi_m=0.\] The vector $\kappa^{\varepsilon}$ belongs to 
\[\widetilde{\Gamma}_{k}^{'}=\left\{\kappa \in \Gamma_k: \frac{\psi_0}{2} \leq \sigma_k(\kappa) \leq 2\psi_1\right\}\]
if $\varepsilon=\varepsilon(\kappa,\psi_0,\psi_1)>0$ is taken to be small enough. The inequality for $m\geq 2$ then follows.
\end{proof}

In practical applications, the constants $\psi_0$ and $\psi_1$ would be taken to depend on the right-hand side of our equation.
\subsection{The elliptic equation and admissible solutions}
\indent

For a symmetric matrix $A=(a_{ij})$ and a function
\[F: \{\text{symmetric matrices}\} \to \bR,\] we define
\begin{equation}
F^{ij}(A)=\frac{\partial F}{\partial a_{ij}}, \quad F^{ij,rs}(A)=\frac{\partial^2 F}{\partial a_{ij} \partial a_{rs}}.\label{derivatives of F}
\end{equation}
When $F$ is of the form
\[F(A)=f(\lambda(A))\] for some symmetric function $f$ of $n$ variables, where $\lambda(A)=(\lambda_1,\ldots,\lambda_n)$ are the eigenvalues of $A$, the function $F$ is as smooth as $f$ and is concave if $f$ is concave. 
When $A$ is diagonal, we have $F^{ij}=f_i\delta_{ij}$ where
\begin{equation}
f_i=\frac{\partial f}{\partial \lambda_i}.\label{derivatives of f}
\end{equation} Moreover, we have
\[\sum_{i,j} F^{ij}a_{ij}=\sum_{i=1}^{n} f_i(\lambda(A))\lambda_i, \quad \sum_{i,j,k} F^{ij}a_{ik}a_{jk}=\sum_{i=1}^{n} f_i(\lambda(A))\lambda_{i}^2.\]

In this note, we are considering equations of exactly this form, where the function 
\[f(\kappa)=H_{k}^{1/k}(\kappa)=\left[\binom{n}{k}^{-1}\sigma_k(\kappa)\right]^{1/k}\]
and the matrix $A=\{a_{ij}\}$ is given by
\[a_{ij}=\frac{1}{w}(\delta_{ij}+u\gamma^{ik}u_{kl}\gamma^{lj}),\] where
\[w=\sqrt{1+|Du|^2}, \quad \gamma_{ij}=\delta_{ij}+\frac{u_iu_j}{1+w}, \quad \gamma^{ij}=\delta_{ij}-\frac{u_iu_j}{w(1+w)}.\] The eigenvalues of the matrix $A$ are the hyperbolic principal curvatures $\kappa[\Sigma]$ of the graph $\Sigma=(x,u(x))$. The prescribed curvature relation \eqref{req1} translates to a fully nonlinear equation of the form
\[G(D^2u,Du,u)=\sigma \quad \text{in $\Omega$},\] which is elliptic and concave in $D^2u$ at $k$-admissible solutions.

\begin{definition}
We say a smooth hypersurface $\Sigma$ in $\bH^{n+1}$ is $k$-admissible if its hyperbolic principal curvatures $\kappa[\Sigma] \in \Gamma_k$ everywhere on $\Sigma$. Moreover, in the present setting, we say a smooth, positive function $u$ is $k$-admissible in a domain $\Omega$ if the vertical graph $\Sigma=(x,u(x))$ over $\Omega$ is  a smooth $k$-admissible hypersurface in $\bH^{n+1}$.
\end{definition}

\section{The curvature estimate for $2k \leq n$} \label{2k<=n proof}
In this section, we prove Theorem \ref{the theorem} for the case when
\[\frac{3-\sqrt{5}}{2}<\frac{2k}{n} \leq 1.\]
As we have discussed in the Introduction, it is sufficient to derive, for $k$-admissible graphs, $\varepsilon$-independent curvature estimates that hold for all $\sigma \in (0,1)$; the rest will follow from the general framework of Guan-Spruck \cite{JEMS}. 

The proof consists of several ingredients: 
\begin{enumerate}
\item A concavity inequality due to Yan \cite{Yan} for the $\sigma_k$ operator in $\Gamma_k$ which works well for $2k>n$ but loses its effectiveness when $2k \leq n$;
\item A concavity inequality due to Zhang \cite{Zhang} for the $\sigma_k$ operator that can be effective for all $k$ but assumes a semi-convexity condition;
\item The general derivation of Lu \cite{Lu-PAMS} which assumes a concavity inequality of the Ren-Wang \cite{Ren-Wang-1, Ren-Wang-2} type is available;
\item A reduction technique developed by us in \cite{Bin-Adv}.
\end{enumerate}

With Yan's inequality, we apply our reduction technique to show that the semi-convexity condition can indeed be achieved at the maximum point. Then, by the semi-convexity condition, we apply Zhang's concavity inequality which is of the Ren-Wang type to eliminate the largely negative third order term. After that, we carry out Lu's argument to eliminate another largely negative term. This last step is the only place where we will have to impose the condition
\[\frac{3-\sqrt{5}}{2}<\frac{2k}{n} \leq 1,\] otherwise our approach would work for all $2k \leq n$ in general.

For the proof, in order to conveniently apply the concavity inequalities, we work with the $\sigma_k$ operator instead of $H_{k}^{1/k}$.

\begin{theorem}
Let $n \geq 3$ and $2 \leq k \leq n$ satisfy 
\[\frac{3-\sqrt{5}}{2}<\frac{2k}{n} \leq 1.\]
Suppose $\Omega \subseteq \bR^n$ is a smooth bounded domain with a mean-convex boundary and $\sigma \in (0,1)$. For $f=H_{k}^{1/k}$ and $\Gamma=\Gamma_{k}$ in \eqref{req1}-\eqref{req2}, let $u$ be a smooth solution of the approximate Dirichlet problem \eqref{the approximate problem} whose vertical graph $\Sigma=(x,u(x))$ over the domain $\Omega$ is $k$-admissible. Then its largest principal curvature
\[\kappa_{\max}(x):=\max_{1 \leq i \leq n} \kappa_i(x)\] satisfies
\[\max_{\Omega} \kappa_{\max} \leq C \left(1+ \max_{\partial \Omega}\kappa_{\max}\right)\] for some universal constant $C>0$ that depends only on $n$, $k$ and $\sigma$, but is independent of $\varepsilon$.
\end{theorem}

\begin{proof}
We consider the test function
\[Q= \frac{\kappa_{\max}}{(\nu^{n+1})^{N}},\] where $N>1$ is some constant to be determined later. By the mean-convexity assumption on $\partial \Omega$, we have \cite[Proposition 4.1]{JEMS} that $\nu^{n+1} \geq \sigma$ everywhere on $\Sigma$.

Suppose $Q$ attains its maximum at some interior $X_0 \in \Sigma$. Let $\{\tau_1,\ldots,\tau_n\}$ be a local orthonormal frame around $X_0$ such that the second fundamental form $h_{ij}(X_0)=\kappa_i(X_0)\delta_{ij}$ is diagonal and the principal curvatures are ordered as
\[\kappa_{\max}(X_0)=\kappa_1 (X_0) \geq  \cdots \geq \kappa_n(X_0).\] 
In case $\kappa_1$ has multiplicity $m>1$, i.e.,
\[\kappa_1=\cdots=\kappa_m>\kappa_{m+1}\geq \cdots \geq \kappa_n,\] we can define a smooth function $\varphi(X)$ such that
\[Q(X_0)=\frac{\varphi(X)}{(\nu^{n+1}(X))^N}.\] It follows that $\varphi \geq \kappa_1$ everywhere and $\varphi=\kappa_1$ at $X_0$. By a smooth approximation lemma \cite[Lemma 5]{BCD} due to Brendle-Choi-Daskalopoulos, we have
\begin{align}
\delta_{kl}\cdot \nabla_i\ \varphi &=\nabla_i h_{kl}, \quad 1 \leq k,l \leq m, \label{1st order approximation}\\
\nabla_i\nabla_i\ \varphi&\geq \nabla_i\nabla_ih_{11}+2\sum_{p>m}\frac{(\nabla_ih_{1p})^2}{\kappa_1-\kappa_p}, \nonumber
\end{align} at the point $X_0$. From here onwards, all calculations will be done at the point $X_0$ without explicitly saying so, and we will use standard shorthand notations for the covariant derivatives, e.g., $h_{ijk}=\nabla_k h_{ij}$, $h_{ijkl}=\nabla_l\nabla_k h_{ij}$ and so on.
Now, since the smooth function
\[\tilde{Q}(X):=\frac{\varphi(X)}{(\nu^{n+1})^N}\] has the constant value $Q(X_0)$, we have that
\begin{align}
0=(\log \tilde{Q})_i&=\frac{\varphi_i}{\varphi} - N\frac{\nabla_i\nu^{n+1}}{\nu^{n+1}} \nonumber \\
&= \frac{h_{11i}}{\kappa_1}-N\frac{\nabla_i \nu^{n+1}}{\nu^{n+1}}, \label{1st order critical}
\end{align} and
\begin{align}
0=&\ (\log \tilde{Q})_{ii} \nonumber\\
=&\ \frac{\varphi_{ii}}{\varphi}-\frac{\varphi_{i}^2}{\varphi^2} - N\frac{\nabla_{ii}\nu^{n+1}}{\nu^{n+1}} + N \frac{(\nabla_{i}\nu^{n+1})^2}{(\nu^{n+1})^2} \nonumber\\
\geq &\ \frac{h_{11ii}}{\kappa_1}+2\sum_{p>m}\frac{h_{1pi}^2}{\kappa_1(\kappa_1-\kappa_p)}-\frac{h_{11i}^2}{\kappa_{1}^2}- N\frac{\nabla_{ii}\nu^{n+1}}{\nu^{n+1}} + N \frac{(\nabla_{i}\nu^{n+1})^2}{(\nu^{n+1})^2}. \label{2nd order critical 1}
\end{align} 

\begin{remark}
In case $m=1$, we would immediately get \eqref{1st order critical} and \eqref{2nd order critical 1}.
\end{remark}

Contracting \eqref{2nd order critical 1} with $F^{ii}=\sigma_{k}^{ii}$, we have that
\begin{gather} \label{2nd order critical 2}
\begin{split}
0 &\geq  \sum_{i=1}^{n} \frac{F^{ii}h_{11ii}}{\kappa_1}+2\sum_{i=1}^{n} \sum_{p>m}\frac{F^{ii}h_{1pi}^2}{\kappa_1(\kappa_1-\kappa_p)}-\sum_{i=1}^{n} \frac{F^{ii}h_{11i}^2}{\kappa_{1}^2}  \\
&\quad - \frac{N}{\nu^{n+1}}\sum_{i=1}^{n} F^{ii}\nabla_{ii}\nu^{n+1}+N\sum_{i=1}^{n} \frac{F^{ii}(\nabla_i\nu^{n+1})^2}{(\nu^{n+1})^2}.
\end{split}
\end{gather} We shall first deal with the term $h_{11ii}$. Since $\bH^{n+1}$ has constant sectional curvature $-1$, by the Codazzi and Gauss equations, we have $h_{ijk}=h_{ikj}$ and 
\[h_{11ii}=h_{ii11}+h_{11}^2h_{ii}-h_{11}h_{ii}^2-h_{11}+h_{ii}.\] Hence,
\begin{equation}
\sum_{i=1}^{n} F^{ii}h_{11ii}=\sum_{i=1}^{n} F^{ii}h_{ii11}-\kappa_1\left(\sum_{i=1}^{n} F^{ii}\kappa_{i}^2 + \sum_{i=1}^{n} F^{ii} \right) + (\kappa_{1}^2+1)kF,\label{expansion of the fourth order term}
\end{equation} where we have used Lemma \ref{sigma_k properties 1} to get
\[
\sum_{i=1}^{n} F^{ii}h_{ii}=\sum_{i=1}^{n} \sigma_{k}^{ii}\kappa_i=k\sigma_k.
\] 
Next, to handle the term $h_{ii11}$, we differentiate the curvature equation twice, we obtain that
\begin{equation}
\sum_{i=1}^{n} F^{ii}h_{ii1}=0 \label{differentiate once}
\end{equation} and
\begin{equation}
\sum_{p,q,r,s} F^{pq,rs}h_{pq1}h_{rs1} + \sum_{i=1}^{n} F^{ii}h_{ii11}=0. \label{differentiate twice}
\end{equation}
Substituting \eqref{expansion of the fourth order term} and \eqref{differentiate twice} back into \eqref{2nd order critical 2}, we obtain that
\begin{gather} \label{2nd order critical 3}
\begin{split}
0 &\geq  -\sum_{p,q,r,s} \frac{F^{pq,rs}h_{pq1}h_{rs1}}{\kappa_1} - \left(\sum_{i=1}^{n} F^{ii}\kappa_{i}^2+\sum_{i=1}^{n} F^{ii}\right) + C(n,k,\sigma) \kappa_1\\
&\quad +2\sum_{i=1}^{n} \sum_{p>m}\frac{F^{ii}h_{1pi}^2}{\kappa_1(\kappa_1-\kappa_p)}-\sum_{i=1}^{n} \frac{F^{ii}h_{11i}^2}{\kappa_{1}^2}  \\
&\quad - \frac{N}{\nu^{n+1}}\sum_{i=1}^{n} F^{ii}\nabla_{ii}\nu^{n+1}+N\sum_{i=1}^{n} \frac{F^{ii}(\nabla_i\nu^{n+1})^2}{(\nu^{n+1})^2}.
\end{split}
\end{gather} Using \eqref{2nd derivatives for sigma k} and \eqref{2nd derivative for sigma k 2} along with the symmetry and the Codazzi properties of $h_{ij}$, we may further expand the $h_{ii11}$ term as
\begin{gather} \label{expansion 1}
\begin{split}
\sum_{i=1}^{n} F^{ii}h_{ii11}&=-\sum_{p,q,r,s} F^{pq,rs}h_{pq1}h_{rs1} \\
&=-\sum_{p\neq q}F^{pp,qq}h_{pp1}h_{qq1}+\sum_{p \neq q}F^{pp,qq}h_{pq1}^2 \\
&\geq -\sum_{p\neq q}F^{pp,qq}h_{pp1}h_{qq1}+ 2 \sum_{i>m}\frac{F^{ii}-F^{11}}{\kappa_1-\kappa_i}h_{11i}^2,
\end{split}
\end{gather} where we have also used \eqref{1st order approximation} to deduce that
\[h_{i11}=\delta_{i1}\cdot \varphi_1 = 0 \quad \text{for $1<i \leq m$}.\] On the other hand, we may expand the $h_{1pi}$ term as follows,
\begin{gather} \label{expansion 2}
\begin{split}
2\sum_{i=1}^{n} \sum_{p>m}\frac{F^{ii}h_{1pi}^2}{\kappa_1(\kappa_1-\kappa_p)} & \geq 2\sum_{p>m}\frac{F^{pp}h_{1pp}^2}{\kappa_1(\kappa_1-\kappa_p)}+2\sum_{p>m}\frac{F^{11}h_{1p1}^2}{\kappa_1(\kappa_1-\kappa_p)}\\
&=2\sum_{i>m}\frac{F^{ii}h_{ii1}^2}{\kappa_1(\kappa_1-\kappa_i)}+2\sum_{i>m}\frac{F^{11}h_{11i}^2}{\kappa_1(\kappa_1-\kappa_i)}.
\end{split}
\end{gather} Substituting \eqref{expansion 1} and \eqref{expansion 2} back into \eqref{2nd order critical 3}, we have
\begin{gather} \label{2nd order critical 4}
\begin{split}
0&\geq -\sum_{p \neq q} \frac{F^{pp,qq}h_{pp1}h_{qq1}}{\kappa_1}+2\sum_{i>m}\frac{F^{ii}h_{ii1}^2}{\kappa_1(\kappa_1-\kappa_i)}-\left(\sum_{i=1}^{n} F^{ii} + \sum_{i=1}^{n} F^{ii}\kappa_{i}^2\right)\\
&\quad +2\sum_{i>m}\frac{F^{ii}-F^{11}}{\kappa_1(\kappa_1-\kappa_i)}h_{11i}^2+2\sum_{i>m}\frac{F^{11}h_{11i}^2}{\kappa_1(\kappa_1-\kappa_i)}-\sum_{i=1}^{n} \frac{F^{ii}h_{11i}^2}{\kappa_{1}^2}\\
&\quad -\frac{N}{\nu^{n+1}}\sum_{i=1}^{n} F^{ii}\nabla_{ii} \nu^{n+1}+N\sum_{i=1}^{n} \frac{F^{ii}(\nabla_i\nu^{n+1})^2}{(\nu^{n+1})^2}+C(n,k,\sigma)\kappa_1.\\
\end{split}
\end{gather} 
For the $\nabla_{ii}\nu^{n+1}$ term, we proceed by a standard computation \cite[Lemma 2.1]{Lu-PAMS} to obtain that
\begin{align*}
&\ \sum_{i=1}^{n} F^{ii} \nabla_{ii}\nu^{n+1}\\
 =&\ 2\sum_{i=1}^{n} F^{ii}\frac{u_i}{u}\nabla_i\nu^{n+1}+kF[1+(\nu^{n+1})^2]\\
 & -\nu^{n+1}\left(\sum_{i=1}^{n} F^{ii}+\sum_{i=1}^{n} F^{ii}\kappa_{i}^2\right).
\end{align*} For the third order term $h_{11i}$, we expand
\[\sum_{i=1}^{n} \frac{F^{ii}h_{11i}^2}{\kappa_{1}^2}=\frac{F^{11}h_{111}^2}{\kappa_{1}^2}+\sum_{i>m}\frac{F^{ii}h_{11i}^2}{\kappa_{1}^2}\] by noting from \eqref{1st order approximation} that
\begin{equation}
h_{11i}=h_{1i1}=\delta_{1i}\cdot (\varphi)_{1} = 0 \quad \text{for $1<i \leq m$}. \label{approximation 2}
\end{equation}
\begin{remark}
The expansion is trivial when $m=1$.
\end{remark}
Therefore, from \eqref{2nd order critical 4}, we get
\begin{gather} \label{2nd order critical 5}
\begin{split}
0&\geq -\sum_{p \neq q} \frac{F^{pp,qq}h_{pp1}h_{qq1}}{\kappa_1} - \frac{F^{11}h_{111}^2}{\kappa_{1}^2}+2\sum_{i>m}\frac{F^{ii}h_{ii1}^2}{\kappa_1(\kappa_1-\kappa_i)} \\
&\quad +2\sum_{i>m}\frac{F^{ii}h_{11i}^2}{\kappa_1(\kappa_1-\kappa_i)}-\sum_{i>m}\frac{F^{ii}h_{11i}^2}{\kappa_{1}^2}+N\sum_{i=1}^{n} \frac{F^{ii}(\nabla_i\nu^{n+1})^2}{(\nu^{n+1})^2}\\
&\quad +(N-1)\left(\sum_{i=1}^{n} F^{ii} + \sum_{i=1}^{n} F^{ii}\kappa_{i}^2\right)-2N\sum_{i=1}^{n} F^{ii}\frac{u_i}{u}\frac{\nabla_i\nu^{n+1}}{\nu^{n+1}}\\
&\quad +C(n,k,\sigma)\kappa_1-C(n,k,\sigma) N.
\end{split}
\end{gather} 

Let us now pause for a moment to discuss the difficulties.

\subsection{The difficulties} \label{the difficulties}
\indent

In contrast to the Euclidean case, here we have an additional term,
\begin{equation}
-2N\sum_{i=1}^{n} F^{ii}\frac{u_i}{u}\frac{\nabla_i\nu^{n+1}}{\nu^{n+1}},\label{the subtlety 1}
\end{equation} which would cause substantial difficulties as we now explain. We may employ some of the terms in \eqref{2nd order critical 5} to have the following control
\begin{gather}\label{control the subtlety}
\begin{split}
&\ 2\sum_{i>m}\frac{F^{ii}h_{11i}^2}{\kappa_1(\kappa_1-\kappa_i)}-\sum_{i>m}\frac{F^{ii}h_{11i}^2}{\kappa_{1}^2}-2N\sum_{i=1}^{n} F^{ii}\frac{u_i}{u}\frac{\nabla_i\nu^{n+1}}{\nu^{n+1}} \\
\geq &\ -C\sum_{i=1}^{n} F^{ii}.
\end{split}
\end{gather}
In order to eliminate this term which could be largely negative, one natural thought would be to take a large $N>0$ in the third line of \eqref{2nd order critical 5}. However, this would lead to another trouble,
\begin{equation}
-\frac{F^{11}h_{111}^2}{\kappa_{1}^2} \sim -CN^2F^{11}\kappa_{1}^2 \quad \text{by \eqref{1st order critical}},\label{the subtlety 2}
\end{equation} where the coefficient $N^2$ would be even larger and the term
\begin{equation}
(N-1)F^{11}\kappa_{1}^2 \label{control the subtlety 2}
\end{equation} would not be enough for absorbing it.

\begin{remark} \label{the degeneracy}
One may also consider adding terms to the test function $Q$ so that it would produce a term of 
\[+C\sum_{i=1}^{n} F^{ii}.\] The problem is, most terms that could give rise to such a term will make the constant $C>0$ in \eqref{global C2 curvature} depend on 
\[\frac{1}{\varepsilon};\] see \eqref{global C2} for example. This would not allow us to the passage of limits for solving \eqref{the Dirichlet problem}. The degeneracy of the equation crucially causes such difficulties.
\end{remark}

This issue can be perfectly resolved when $2k>n$, by following Lu's argument in \cite{Lu-PAMS}. First, when $2k>n$, the concavity inequality \eqref{concavity 1} is as powerful as those of Ren-Wang \cite{Ren-Wang-1, Ren-Wang-2}, which yields
\begin{equation}
-\sum_{p \neq q} \frac{F^{pp,qq}h_{pp1}h_{qq1}}{\kappa_1} - \frac{F^{11}h_{111}^2}{\kappa_{1}^2}+2\sum_{i>m}\frac{F^{ii}h_{ii1}^2}{\kappa_1(\kappa_1-\kappa_i)} \geq 0, \label{concavity 3}
\end{equation} and so we would not need to worry about \eqref{the subtlety 2}. Secondly, by Lemma \ref{sigma_k properties 2}, we can derive that
\begin{equation}
2\sum_{i>m}\frac{F^{ii}h_{11i}^2}{\kappa_1(\kappa_1-\kappa_i)}-\sum_{i>m}\frac{F^{ii}h_{11i}^2}{\kappa_{1}^2} \geq \frac{2k-n}{n} \sum_{i>m} \frac{F^{ii}h_{11i}^2}{\kappa_{1}^2}. \label{the subtlety 3}
\end{equation} When $2k>n$, we would have a positive lower bound on the right-hand side of \eqref{the subtlety 3} and only then the control \eqref{control the subtlety} would follow. Finally, the desired estimate would readily follow by taking $N>0$ large in the third line of \eqref{2nd order critical 5}. This summarizes Lu's proof strategy in \cite{Lu-PAMS}.

Now, when $2k=n$, by the allowable range in \eqref{gamma}, we can still apply \eqref{concavity 1} to get the following slightly weaker control of \eqref{the subtlety 2},
\[-\sum_{p \neq q} \frac{F^{pp,qq}h_{pp1}h_{qq1}}{\kappa_1} - \frac{F^{11}h_{111}^2}{\kappa_{1}^2}+2\sum_{i>m}\frac{F^{ii}h_{ii1}^2}{\kappa_1(\kappa_1-\kappa_i)} \geq -\varepsilon \frac{F^{11}h_{111}^2}{\kappa_{1}^2}\] for a small $\varepsilon>0$. This can be suitably controlled by \eqref{control the subtlety 2} because now we have a small coefficient $\varepsilon$ in front of \eqref{the subtlety 2}. However, when $2k=n$, the lower bound for \eqref{the subtlety 3} degenerates to $0$ and we would not be able to use it for controlling \eqref{the subtlety 1}.

When $2k<n$, the situation is even worse: The concavity inequality \eqref{concavity 1} totally loses its effectiveness because of the allowable range \eqref{gamma} and the lower bound for \eqref{the subtlety 3} becomes negative which provides more negativity in addition to the trouble term \eqref{the subtlety 1}.

In other words, Lu's proof in \cite{Lu-PAMS} would not be suitable for handling the case $2k \leq n$. In order to get the estimate for that case, we have to supply new elements in accordance with Lu's proof. We shall introduce the novel reduction technique that was developed in Wang's investigation \cite{Bin-Adv} for the case $k=2, n \geq 4$. This technique would enable us to show that there is a lower bound for the principal curvatures,
\[\kappa_{n} \geq -C(n,k,\sigma,N) \quad \text{if we choose a suitable, non-large value for $N$}.\]

This reduces the proof to a ``semi-convex'' situation, in which case a powerful concavity inequality in the form \eqref{concavity 3} of Ren-Wang \cite{Ren-Wang-1, Ren-Wang-2} can hold true for all $k$. Moreover, with this curvature control from below, we would be able to retain a positive lower bound for \eqref{the subtlety 3} as well. Hence, the control \eqref{control the subtlety} would follow and our non-large choice of $N$ would happen to be ``large enough'' so that
\[(N-1)\sum_{i=1}^{n} F^{ii} - C\sum_{i=1}^{n} F^{ii} \geq 0.\] This complements Lu's argument in \cite{Lu-PAMS} and the proof therein can now be carried over to the degenerate cases $2k \leq n$.

Let us now proceed to show the curvature lower bound as claimed, by following Wang's argument in \cite{Bin-Adv}.

\subsection{The curvature lower bound}
\indent

We devote this subsection to proving the following claim.
\begin{lemma} \label{the claim}
When $2k \leq n$, we can find some constant $N=N(n,k)>1$ such that if the test function $Q$ attains its maximum at some interior $X_0 \in \Sigma$, then there exists some constant $C=C(n,k,\sigma,N)>0$ such that
\[\kappa_{n}(X_0) \geq -C.\]
\end{lemma}
\begin{proof}[Proof of Lemma \ref{the claim}]

If $\kappa_n(X_0) \geq 0$, then the bound is trivial. Assume $\kappa_n(X_0)<0$. The key to our reduction technique is to utilize the term
\begin{equation}
N\sum_{i=1}^{n} \frac{F^{ii}(\nabla_i\nu^{n+1})^2}{(\nu^{n+1})^2} \label{the particular term}
\end{equation} in the second line of \eqref{2nd order critical 5}.

For convenience, we recall the inequality \eqref{2nd order critical 5} here,
\begin{gather} \label{2nd order critical 6}
\begin{split}
0&\geq -\sum_{p \neq q} \frac{F^{pp,qq}h_{pp1}h_{qq1}}{\kappa_1} - \frac{F^{11}h_{111}^2}{\kappa_{1}^2}+2\sum_{i>m}\frac{F^{ii}h_{ii1}^2}{\kappa_1(\kappa_1-\kappa_i)} \\
&\quad +2\sum_{i>m}\frac{F^{ii}h_{11i}^2}{\kappa_1(\kappa_1-\kappa_i)}-\sum_{i>m}\frac{F^{ii}h_{11i}^2}{\kappa_{1}^2}+N\sum_{i=1}^{n} \frac{F^{ii}(\nabla_i\nu^{n+1})^2}{(\nu^{n+1})^2}\\
&\quad +(N-1)\left(\sum_{i=1}^{n} F^{ii} + \sum_{i=1}^{n} F^{ii}\kappa_{i}^2\right)-2N\sum_{i=1}^{n} F^{ii}\frac{u_i}{u}\frac{\nabla_i\nu^{n+1}}{\nu^{n+1}}\\
&\quad +C(n,k,\sigma)\kappa_1-C(n,k,\sigma) N.
\end{split}
\end{gather} For the first line, by Lemma \ref{the concavity inequality 1} and \eqref{1st order approximation}, there is some constant $\gamma_0=\gamma_0(n,k)>0$ such that
\[-\sum_{p \neq q} \frac{F^{pp,qq}h_{pp1}h_{qq1}}{\kappa_1} - \frac{F^{11}h_{111}^2}{\kappa_{1}^2}+2\sum_{i>m}\frac{F^{ii}h_{ii1}^2}{\kappa_1(\kappa_1-\kappa_i)} \geq (\gamma_0-1)\frac{F^{11}h_{111}^2}{\kappa_{1}^2}.\] Adding this to the $i=1$ summand of \eqref{the particular term} and invoking the first order critical condition \eqref{1st order critical}, we get
\begin{equation}
(\gamma_0-1)\frac{F^{11}h_{111}^2}{\kappa_{1}^2}+N\frac{F^{11}(\nabla_1\nu^{n+1})^2}{(\nu^{n+1})^2}=\left(\gamma_0-1+\frac{1}{N}\right)\frac{F^{11}h_{111}^2}{\kappa_{1}^2}. \label{the subtlety 4}
\end{equation}
\begin{remark}
The term \eqref{the particular term} was plausibly omitted in Lu's argument \cite{Lu-PAMS} because the coefficient $1/N$ would be too small when $N$ is large. In order to utilize this term, we will choose $N$ to be neither large nor small, but a modest value such that it would gauge everything perfectly.
\end{remark}
Indeed, let us first choose
\begin{equation}
1<N \leq \frac{1}{1-\gamma_0} \label{choice of N 1}
\end{equation} so that \eqref{the subtlety 4} is non-negative. For the second line in \eqref{2nd order critical 6}, we proceed similarly as follows. By Lemma \ref{sigma_k properties 2} and the first order critical condition \eqref{1st order critical},
\begin{gather} \label{the subtlety 5}
\begin{split}
&\ 2\sum_{i>m}\frac{F^{ii}h_{11i}^2}{\kappa_1(\kappa_1-\kappa_i)}-\sum_{i>m}\frac{F^{ii}h_{11i}^2}{\kappa_{1}^2}+N\sum_{i>m}^{n} \frac{F^{ii}(\nabla_i\nu^{n+1})^2}{(\nu^{n+1})^2} \\
=&\ \sum_{i>m} \frac{\kappa_1+\kappa_i}{\kappa_1-\kappa_i}\frac{F^{ii}h_{11i}^2}{\kappa_{1}^2}+N\sum_{i>m}^{n} \frac{F^{ii}(\nabla_i\nu^{n+1})^2}{(\nu^{n+1})^2}\\
\geq &\ \sum_{i>m} \left(\frac{2k-n}{n}+\frac{1}{N}\right) \frac{F^{ii}h_{11i}^2}{\kappa_{1}^2},
\end{split}
\end{gather} which can be made non-negative if we choose
\begin{equation}
\begin{cases}
1<N \leq \frac{n}{n-2k} &\quad \text{when $2k<n$},\\
\text{any} \quad 1<N     &\quad \text{when $2k=n$}.
\end{cases}
\label{choice of N 2}
\end{equation} Due to the allowable range of $\gamma_0$ in \eqref{gamma}, the choice \eqref{choice of N 1} already implies \eqref{choice of N 2} when $2k< n$. Hence, by choosing a suitable value for $N=N(n,k)>1$ that satisfies \eqref{choice of N 1}, it would follow that both \eqref{the subtlety 4} and \eqref{the subtlety 5} are non-negative. From \eqref{2nd order critical 6}, we are then left with
\begin{gather} \label{2nd order critical 7}
\begin{split}
0&\geq (N-1)\left(\sum_{i=1}^{n} F^{ii} + \sum_{i=1}^{n} F^{ii}\kappa_{i}^2\right)-2N\sum_{i=1}^{n} F^{ii}\frac{u_i}{u}\frac{\nabla_i\nu^{n+1}}{\nu^{n+1}}\\
&\quad +C(n,k,\sigma)\kappa_1-C(n,k,\sigma) N.
\end{split}
\end{gather}

For the trouble term \eqref{the subtlety 1}, we do the following computation,
\begin{align*}
&\ \frac{N-1}{2} \sum_{i=1}^{n} F^{ii}\kappa_{i}^2 -2N\sum_{i=1}^{n} F^{ii}\frac{u_i}{u}\frac{\nabla_i\nu^{n+1}}{\nu^{n+1}}\\
=&\ \frac{N-1}{2} \sum_{i=1}^{n} F^{ii}\kappa_{i}^2 +2N\sum_{i=1}^{n} F^{ii}\frac{u_{i}^2}{u^2}\frac{\kappa_i-\nu^{n+1}}{\nu^{n+1}} \\
\geq &\ \frac{N-1}{2} \sum_{\kappa_i<\nu^{n+1}} F^{ii}\kappa_{i}^2 +2N\sum_{\kappa_i<\nu^{n+1}} F^{ii}\frac{\kappa_i-\nu^{n+1}}{\nu^{n+1}} \\
\geq &\ \sum_{\kappa_i < \nu^{n+1}} F^{ii} \left[\frac{N-1}{2}\kappa_{i}^2 + \frac{2N}{\sigma}\kappa_i-\frac{2N}{\sigma}\right],
\end{align*}
where we have used \eqref{geometric formula} to get the following
\[\nabla_i\nu^{n+1}=\frac{u_{i}}{u}(\nu^{n+1}-\kappa_i), \quad \frac{u_{i}^2}{u^2} \leq \sum_{i=1}^{n} \frac{u_{i}^2}{u^2} \leq 1.\] The quadratic polynomial in the square bracket is non-negative if 
\[\kappa_{i} \leq -\frac{2N+2\sqrt{N[N+(N-1)\sigma]}}{(N-1)\sigma}=:-\eta(N,\sigma).\] In other words, we can reduce the sum as
\begin{align*}
&\ \frac{N-1}{2} \sum_{i=1}^{n} F^{ii}\kappa_{i}^2 -2N\sum_{i=1}^{n} F^{ii}\frac{u_i}{u}\frac{\nabla_i\nu^{n+1}}{\nu^{n+1}} \\
\geq &\ 2N\sum_{-\eta < \kappa_i<\nu^{n+1}} F^{ii}\frac{\kappa_i-\nu^{n+1}}{\nu^{n+1}}.
\end{align*} Looking back to \eqref{2nd order critical 7}, we now have that
\begin{gather} \label{2nd order critical 8}
\begin{split}
0& \geq \frac{N-1}{2}\sum_{i=1}^{n} F^{ii}\kappa_{i}^2 + 2N\sum_{-\eta < \kappa_i<\nu^{n+1}} F^{ii}\frac{\kappa_i-\nu^{n+1}}{\nu^{n+1}} \\
&\quad +C(n,k,\sigma) \kappa_1-C(n,k,\sigma) N \\
& \geq \frac{N-1}{2} F^{nn}\kappa_{n}^2 -\frac{2N}{\sigma}(\eta+1)\sum_{i=1}^{n} F^{ii}+C\kappa_1-CN.
\end{split}
\end{gather}
Since $\kappa_n<0$, we have that 
\[F^{nn} \geq C(n,k) \sum_{i=1}^{n} F^{ii} \quad \text{by Lemma \ref{sigma_k properties 2}}.\] Also, by Lemma \ref{sigma_k properties 1} and Lemma \ref{NM}, we have that
\[\sum_{i=1}^{n} F^{ii}=(n-k+1)\sigma_{k-1}(\kappa) \geq C(n,k) \sigma_{k}^{\frac{k-1}{k}} \geq C(n,k, \sigma).\] Hence, it follows from the inequality \eqref{2nd order critical 8} that
\[0 \geq \frac{N-1}{2} C(n,k)\sum_{i=1}^{n} F^{ii} \cdot \kappa_{n}^2 - \frac{2N}{\sigma}(\eta+1)\sum_{i=1}^{n} F^{ii} - CN\sum_{i=1}^{n} F^{ii}.\] Rearranging the terms and dividing by the positive sum $\sum_{i=1}^{n} F^{ii}$ would yield that
\[\kappa_{n}^2 \leq C(n,k,\sigma,N).\] The claim is now proved.

\end{proof}

\subsection{The proof continued}\label{the proof continued}
\indent

We now continue the proof of the curvature estimate. We restate the inequality \eqref{2nd order critical 5} here.
\begin{gather} \label{2nd order critical 9}
\begin{split}
0&\geq -\sum_{p \neq q} \frac{F^{pp,qq}h_{pp1}h_{qq1}}{\kappa_1} - \frac{F^{11}h_{111}^2}{\kappa_{1}^2}+2\sum_{i>m}\frac{F^{ii}h_{ii1}^2}{\kappa_1(\kappa_1-\kappa_i)} \\
&\quad +2\sum_{i>m}\frac{F^{ii}h_{11i}^2}{\kappa_1(\kappa_1-\kappa_i)}-\sum_{i>m}\frac{F^{ii}h_{11i}^2}{\kappa_{1}^2}+N\sum_{i=1}^{n} \frac{F^{ii}(\nabla_i\nu^{n+1})^2}{(\nu^{n+1})^2}\\
&\quad +(N-1)\left(\sum_{i=1}^{n} F^{ii} + \sum_{i=1}^{n} F^{ii}\kappa_{i}^2\right)-2N\sum_{i=1}^{n} F^{ii}\frac{u_i}{u}\frac{\nabla_i\nu^{n+1}}{\nu^{n+1}}\\
&\quad +C(n,k,\sigma)\kappa_1-C(n,k,\sigma) N.
\end{split}
\end{gather} 

With the curvature lower bound provided by Lemma \ref{the claim}, we can apply Lemma \ref{the concavity inequality 2} to get the following Ren-Wang type inequality,
\[-\sum_{p \neq q} \frac{F^{pp,qq}h_{pp1}h_{qq1}}{\kappa_1} - \frac{F^{11}h_{111}^2}{\kappa_{1}^2}+2\sum_{i>m}\frac{F^{ii}h_{ii1}^2}{\kappa_1(\kappa_1-\kappa_i)} \geq 0,\] so the first line in \eqref{2nd order critical 9} can be discarded.

We next claim that
\begin{lemma}\label{the claim 2}
If
\[\frac{3-\sqrt{5}}{2}<\frac{2k}{n} \leq 1,\] then there exists a choice of $N=N(n,k)>1$ that is compatible with \eqref{choice of N 1} and \eqref{choice of N 2} such that
\begin{align*}
&\ 2\sum_{i>m}\frac{F^{ii}h_{11i}^2}{\kappa_1(\kappa_1-\kappa_i)}-\sum_{i>m}\frac{F^{ii}h_{11i}^2}{\kappa_{1}^2}+N\sum_{i=1}^{n} \frac{F^{ii}(\nabla_i\nu^{n+1})^2}{(\nu^{n+1})^2}\\
&\ +(N-1)\sum_{i=1}^{n} F^{ii} -2N\sum_{i=1}^{n} F^{ii}\frac{u_i}{u}\frac{\nabla_i\nu^{n+1}}{\nu^{n+1}}\\
\geq &\ 0
\end{align*} as well.
\end{lemma}

\begin{remark}
This is the only place in our proof that does not hold for all $k$ with $2k \leq n$; otherwise our approach would be completely general.
\end{remark}

\begin{proof}[Proof of Lemma \ref{the claim 2}]
With the curvature lower bound proved in Lemma \ref{the claim}, we assume $\kappa_1$ is sufficiently large to get the following,
\begin{align*}
&\ 2\sum_{i>m}\frac{F^{ii}h_{11i}^2}{\kappa_1(\kappa_1-\kappa_i)}-\sum_{i>m}\frac{F^{ii}h_{11i}^2}{\kappa_{1}^2}\\
=&\ \sum_{i>m} \frac{\kappa_1+\kappa_i}{\kappa_1-\kappa_i} \frac{F^{ii}h_{11i}^2}{\kappa_{1}^2}\\
\geq &\ \sum_{i>m} \frac{\kappa_1-C}{\kappa_1+C} \frac{F^{ii}h_{11i}^2}{\kappa_{1}^2}\\
\geq &\ (1-\delta) \sum_{i>m} \frac{F^{ii}h_{11i}^2}{\kappa_{1}^2}
\end{align*} for some small $\delta>0$ to be determined. We then use this term to eliminate the trouble term \eqref{the subtlety 1}. Again, we assume $\kappa_1(X_0) >1 \geq \nu^{n+1}(X_0)$ is sufficiently large, then by the first order critical condition \eqref{1st order critical} and the formula
\[\nabla_i\nu^{n+1}=\frac{u_i}{u}(\nu^{n+1}-\kappa_i),\] we derive the following,
\begin{align}
&\ (1-\delta) \sum_{i>m} \frac{F^{ii}h_{11i}^2}{\kappa_{1}^2}+N\sum_{i=1}^{n} \frac{F^{ii}(\nabla_i\nu^{n+1})^2}{(\nu^{n+1})^2}\nonumber\\
&\quad -2N\sum_{i=1}^{n} F^{ii}\frac{u_i}{u}\frac{\nabla_i\nu^{n+1}}{\nu^{n+1}}\nonumber \\
\geq &\ (1-\delta)N^2\sum_{i>m} \frac{F^{ii}(\nabla_i\nu^{n+1})^2}{(\nu^{n+1})^2}+N\sum_{i>m} \frac{F^{ii}(\nabla_i\nu^{n+1})^2}{(\nu^{n+1})^2}\nonumber \\
&\quad -2N\sum_{i>m} F^{ii}\frac{u_i}{u}\frac{\nabla_i\nu^{n+1}}{\nu^{n+1}}\nonumber\\
=&\ [(1-\delta)N^2+N]\sum_{i>m} F^{ii}\frac{u_{i}^2}{u^2}\left(\frac{\kappa_i-\nu^{n+1}}{\nu^{n+1}}\right)^2\nonumber\\
&\quad +2N\sum_{i>m} F^{ii}\frac{u_{i}^2}{u^2}\frac{\kappa_i-\nu^{n+1}}{\nu^{n+1}}\nonumber \\
=&\ \sum_{i>m} F^{ii}\frac{u_{i}^2}{u^2} (a_0t^2+b_0t), \label{the key estimate}
\end{align}
where we have denoted
\[a_0=(1-\delta)N^2+N, \quad b_0=2N, \quad t=\frac{\kappa_i-\nu^{n+1}}{\nu^{n+1}}.\] In particular, we have
\begin{equation}
a_0t^2+b_0t > -c_0 \label{choice of c0 1}
\end{equation} for 
\begin{equation}
c_0>\frac{b_{0}^2}{4a_0}. \label{choice of c0 2}
\end{equation} In order to have that
\[(N-1)\sum_{i=1}^{n} F^{ii}-c_0\sum_{i=1}^{n} F^{ii} \geq 0,\] we would need that
\[N-1 \geq c_0.\] For such an $N$ to exist, the following must be valid,
\[N-1 > \frac{4N^2}{4[(1-\delta)N^2+N]}=\frac{1}{(1-\delta)+\frac{1}{N}}.\] For a small $\varepsilon>0$ to be determined, we can find some $\delta>0$ such that
\[\frac{1}{(1-\delta)+\frac{1}{N}}<\frac{1}{1+\frac{1}{N}}+\varepsilon.\] That is, it suffices to ensure that
\[N-1 \geq \frac{1}{1+\frac{1}{N}}+\varepsilon,\] which is equivalent to 
\[N^2-(1+\varepsilon)N-(1+\varepsilon) \geq 0.\] Since $N$ is positive, the non-negativity holds if 
\[N \geq \frac{(1+\varepsilon)+\sqrt{(1+\varepsilon)(5+\varepsilon)}}{2}.\] To ensure the compatibility with \eqref{choice of N 1}, we need to check whether
\[\frac{1}{1-\gamma_0} \geq \frac{(1+\varepsilon)+\sqrt{(1+\varepsilon)(5+\varepsilon)}}{2},\] where 
\[0<\gamma_0<\frac{2k}{n}\leq 1.\]

\begin{remark}
Since \eqref{choice of N 1} automatically implies \eqref{choice of N 2} when $2k \leq n$, it suffices to consider only the compatibility with \eqref{choice of N 1}.
\end{remark}

Solving the inequality would yield that
\[0<\varepsilon \leq \frac{3\gamma_0-\gamma_{0}^2-1}{(1-\gamma_0)(2-\gamma_0)}.\] For such an $\varepsilon$ to exist, we need
\[3\gamma_0-\gamma_{0}^2-1>0,\] which is equivalent to
\[\gamma_0>\frac{3-\sqrt{5}}{2}.\] According to the allowable range of $\gamma_0$, in order for such a $\gamma_0$ to exist, we would need that
\[\frac{2k}{n}>\frac{3-\sqrt{5}}{2}.\] This imposes the dimensional restriction stated in the hypothesis of Lemma \ref{the claim 2} and the claim is proved.

\end{proof}

Therefore, by Lemma \ref{the claim} and Lemma \ref{the claim 2}, there only remains in \eqref{2nd order critical 9} that
\[0 \geq C\kappa_1-CN\] and the desired estimate readily follows.

The proof of the curvature estimate is now complete.

\end{proof}

\section{The curvature estimate for $2k>n$} \label{2k>n proof}
In this section, we prove the curvature estimate when $2k>n$. Theorem \ref{the theorem} in the same case would follow from the curvature estimate as a consequence of Guan-Spruck's general framework \cite{JEMS}.

The proof of the curvature estimate when $2k>n$, as we have explained in the Introduction and also in Section \ref{the difficulties}, will follow the argument of Lu in \cite{Lu-PAMS}. Moreover, we may consider only $k \leq n-1$ because the case $k=n$ has been treated in the locally strictly convex theory \cite{SGAR,JDG,Rosenberg-Spruck, JGA, Xiao-A}.
\begin{theorem}
Let $n \geq 3$ and $2 \leq k \leq n-1$ be such that $2k > n$. Suppose $\Omega \subseteq \bR^n$ is a smooth bounded domain with a mean-convex boundary and $\sigma \in (0,1)$. For $f=H_{k}^{1/k}$ and $\Gamma=\Gamma_{k}$ in \eqref{req1}-\eqref{req2}, let $u$ be a smooth solution of the approximate Dirichlet problem \eqref{the approximate problem} whose vertical graph $\Sigma=(x,u(x))$ over the domain $\Omega$ is $k$-admissible. Then its largest principal curvature
\[\kappa_{\max}(x):=\max_{1 \leq i \leq n} \kappa_i(x)\] satisfies
\[\max_{\Omega} \kappa_{\max} \leq C \left(1+ \max_{\partial \Omega}\kappa_{\max}\right)\] for some universal constant $C>0$ that depends only on $n$, $k$ and $\sigma$, but is independent of $\varepsilon$.
\end{theorem}
\begin{proof}
As in section \ref{2k<=n proof}, we still consider the test function
\[Q=\frac{\kappa_{\max}}{(\nu^{n+1})^N}\] but this time we will choose a large $N>0$. Proceeding exactly as in section \ref{2k<=n proof}, we would also arrive at \eqref{2nd order critical 5} which we restate here,
\begin{gather} \label{2nd order critical 10}
\begin{split}
0&\geq -\sum_{p \neq q} \frac{F^{pp,qq}h_{pp1}h_{qq1}}{\kappa_1} - \frac{F^{11}h_{111}^2}{\kappa_{1}^2}+2\sum_{i>m}\frac{F^{ii}h_{ii1}^2}{\kappa_1(\kappa_1-\kappa_i)} \\
&\quad +2\sum_{i>m}\frac{F^{ii}h_{11i}^2}{\kappa_1(\kappa_1-\kappa_i)}-\sum_{i>m}\frac{F^{ii}h_{11i}^2}{\kappa_{1}^2}+N\sum_{i=1}^{n} \frac{F^{ii}(\nabla_i\nu^{n+1})^2}{(\nu^{n+1})^2}\\
&\quad +(N-1)\left(\sum_{i=1}^{n} F^{ii} + \sum_{i=1}^{n} F^{ii}\kappa_{i}^2\right)-2N\sum_{i=1}^{n} F^{ii}\frac{u_i}{u}\frac{\nabla_i\nu^{n+1}}{\nu^{n+1}}\\
&\quad +C(n,k,\sigma)\kappa_1-C(n,k,\sigma) N.
\end{split}
\end{gather} 
There are two negative trouble terms,
\[- \frac{F^{11}h_{111}^2}{\kappa_{1}^2} \quad \text{and} \quad -2N\sum_{i=1}^{n} F^{ii}\frac{u_i}{u}\frac{\nabla_i\nu^{n+1}}{\nu^{n+1}}.\] 

For the first trouble term, since $2k>n$, it follows directly from \eqref{concavity 1} and \eqref{differentiate once} that
\[-\sum_{p \neq q} \frac{F^{pp,qq}h_{pp1}h_{qq1}}{\kappa_1} - \frac{F^{11}h_{111}^2}{\kappa_{1}^2}+2\sum_{i>m}\frac{F^{ii}h_{ii1}^2}{\kappa_1(\kappa_1-\kappa_i)} \geq 0.\]

For the second trouble term, we assume $\kappa_1(X_0)$ is sufficiently large, and use the first order critical condition \eqref{1st order critical} and Lemma \ref{sigma_k properties 2} to derive that
\begin{align*}
&\ 2\sum_{i>m}\frac{F^{ii}h_{11i}^2}{\kappa_1(\kappa_1-\kappa_i)}-\sum_{i>m}\frac{F^{ii}h_{11i}^2}{\kappa_{1}^2}-2N\sum_{i=1}^{n} F^{ii}\frac{u_i}{u}\frac{\nabla_i\nu^{n+1}}{\nu^{n+1}}\\
\geq &\ \sum_{i>m} \frac{\kappa_1+\kappa_i}{\kappa_1-\kappa_i} \frac{F^{ii}h_{11i}^2}{\kappa_{1}^2}-2N\sum_{i>m} F^{ii}\frac{u_i}{u}\frac{\nabla_i\nu^{n+1}}{\nu^{n+1}} \\
\geq &\ \frac{2k-n}{n}\sum_{i>m} F^{ii} N^2\left(\frac{\nabla_i\nu^{n+1}}{\nu^{n+1}}\right)^2-2N\sum_{i>m} F^{ii}\frac{u_i}{u}\frac{\nabla_i\nu^{n+1}}{\nu^{n+1}}\\
\geq &\ \frac{N^2}{n}\sum_{i>m} F^{ii}\frac{u_{i}^2}{u^2}\left(\frac{\kappa_i-\nu^{n+1}}{\nu^{n+1}}\right)^2+2 N\sum_{i>m} F^{ii}\frac{u_{i}^2}{u^2}\frac{\kappa_i-\nu^{n+1}}{\nu^{n+1}}\\
\geq &\ -2n\sum_{i>m} F^{ii}.
\end{align*}

Therefore, by choosing $N$ large enough, e.g., $N=2n+1$, we have that
\[(N-1)\sum_{i=1}^{n} F^{ii}+\frac{1}{n}\sum_{i>m} F^{ii}\frac{h_{11i}^2}{\kappa_{1}^2} - 2 N\sum_{i=1}^{n} F^{ii}\frac{u_i}{u}\frac{\nabla_i \nu^{n+1}}{\nu^{n+1}} \geq 0.\] Hence, from \eqref{2nd order critical 10}, we are left with
\[0 \geq C(n, k, \sigma) \kappa_1-C(n,k,\sigma) N\] and the desired estimate readily follows.

The proof of curvature estimates for $2k > n$ is now complete.
\end{proof}

\begin{remark}
The concavity inequality for the $\sigma_k$ operator may have several other applications. In \cite{Bin-rigidity}, we have applied the concavity inequality to obtain a Liouville type rigidity result for entire solutions of the $k$-Hessian equation when $2k>n$.
\end{remark}

\bibliography{refs}

\appendix

\section{An extension of Lemma \ref{the claim 2} to all $2k \leq n$}

In Section \ref{2k<=n proof}, we established Lemma \ref{the claim 2} under the assumption that
\[\frac{3-\sqrt{5}}{2}<\frac{2k}{n} \leq 1.\] This is the only restriction that prevents Theorem \ref{the theorem} from holding for all $k$.

In \cite{Mei-Yan}, through a different set of arguments, Mei and Yan were able to obtain the curvature estimate \eqref{global C2 curvature} (and hence Theorem \ref{the theorem}) for all $2k \leq n$, without the restriction
\begin{equation}
\frac{3-\sqrt{5}}{2}<\frac{2k}{n}. \label{the restriction 2}
\end{equation}

Being inspired by one step of their argument (which is exactly \eqref{the observation} below), we are now also able to remove the restriction \eqref{the restriction 2} from Lemma \ref{the claim 2} and hence from Theorem \ref{the theorem}. In this appendix, we include this refinement for the potentially interested reader, as our proof method may stay helpful to similar problems.

\begin{proof}[A refinement for Lemma \ref{the claim 2}]
Recall the inequality \eqref{the key estimate} which we restate here,
\begin{align*}
&\ (1-\delta) \sum_{i>m} \frac{F^{ii}h_{11i}^2}{\kappa_{1}^2}+N\sum_{i=1}^{n} \frac{F^{ii}(\nabla_i\nu^{n+1})^2}{(\nu^{n+1})^2}-2N\sum_{i=1}^{n} F^{ii}\frac{u_i}{u}\frac{\nabla_i\nu^{n+1}}{\nu^{n+1}} \\
\geq &\ \sum_{i>m} F^{ii}\frac{u_{i}^2}{u^2}(a_0t^2+b_0t),
\end{align*} where
\[a_0=(1-\delta)N^2+N,\quad b_0=2N, \quad t=\frac{\kappa_i-\nu^{n+1}}{\nu^{n+1}}.\]

The original estimate 
\[\sum_{i>m} F^{ii}\frac{u_{i}^2}{u^2}(a_0t^2+b_0t) \geq -c_0\sum_{i=1}^{n} F^{ii}\] in Section \ref{2k<=n proof} was coarse.
The key is to get a better estimate for this sum, so that we can have
\begin{align*}
&\ 2\sum_{i>m}\frac{F^{ii}h_{11i}^2}{\kappa_1(\kappa_1-\kappa_i)}-\sum_{i>m}\frac{F^{ii}h_{11i}^2}{\kappa_{1}^2}+N\sum_{i=1}^{n} \frac{F^{ii}(\nabla_i\nu^{n+1})^2}{(\nu^{n+1})^2}\\
&\ +(N-1)\sum_{i=1}^{n} F^{ii} -2N\sum_{i=1}^{n} F^{ii}\frac{u_i}{u}\frac{\nabla_i\nu^{n+1}}{\nu^{n+1}}\\
\geq &\ 0.
\end{align*}

Indeed, for each $i$, the quadratic polynomial
\[a_0t^2+b_0t=t(a_0t+b_0)\] is non-negative if $\kappa_i \geq \nu^{n+1}$ or 
\[\kappa_i<\nu^{n+1} \quad \text{and} \quad a_0t+b_0 \leq 0.\] The latter condition corresponds to
\[\kappa_i \leq \nu^{n+1} \left(1-\frac{b_0}{a_0}\right).\] In other words, it remains to deal with the sum
\[\sum_{i \in I} F^{ii}\frac{u_{i}^2}{u^2} (a_0t^2+b_0t),\] where
\[I:=\left\{i>m: \nu^{n+1}\left(1-\frac{b_0}{a_0}\right)<\kappa_i<\nu^{n+1}\right\}.\] By the gradient estimate in \cite[Proposition 4.1]{JEMS}, we know that $\nu^{n+1} \geq \sigma$. On the other hand, 
\[1-\frac{b_0}{a_0}=1-\frac{2N}{(1-\delta)N^2+N}=\frac{(1-\delta)N-1}{(1-\delta)N+1} \geq 0\] if
\begin{equation}
N \geq \frac{1}{1-\delta}. \label{choice of N 3}
\end{equation} It follows that, for $i \in I$, we have
\[\kappa_i>\nu^{n+1} \left(1-\frac{b_0}{a_0}\right) \geq 0.\] By Lemma \ref{sigma_k properties 1}, we have that
\begin{equation}
F^{ii}=\sigma_{k}^{ii} \leq \sigma_{k-1}=\frac{1}{n-k+1}\sum_{j=1}^{n}F^{jj} \quad \text{for $i \in I$}. \label{the observation}
\end{equation} With this extra factor
\begin{equation}
\frac{1}{n-k+1} \label{the extra factor}
\end{equation} which was not considered by us in the proof of Lemma \ref{the claim 2}, we can get the following better estimate,
\begin{align*}
&\ \sum_{i \in I} F^{ii}\frac{u_{i}^2}{u^2} (a_0t^2+b_0t)\\
\geq &\ -c_0 \sum_{i \in I} F^{ii} \frac{u_{i}^2}{u^2} \quad \text{by \eqref{choice of c0 1}-\eqref{choice of c0 2}} \\
\geq &\ -\frac{c_0}{n-k+1}\sum_{j=1}^{n} F^{jj} \cdot \sum_{i \in I} \frac{u_{i}^2}{u^2} \\
\geq &\ -\frac{c_0}{n-k+1}\sum_{i=1}^{n} F^{ii},
\end{align*} where we used
\[\sum_{i \in I} \frac{u_{i}^2}{u^2} \leq \sum_{i=1}^{n} \frac{u_{i}^2}{u^2}=|\tilde{\nabla}u|^2=1-(\nu^{n+1})^2 \leq 1.\]

Now, we want that
\[(N-1)\sum_{i=1}^{n} F^{ii} - \frac{c_0}{n-k+1}\sum_{i=1}^{n} F^{ii} \geq 0,\] that is, we need
\begin{equation}
N - 1 \geq \frac{c_0}{n-k+1} > \frac{b_{0}^2}{4a_0} \frac{1}{n-k+1} = \frac{1}{(1-\delta)+\frac{1}{N}}\frac{1}{n-k+1} \label{choice of N 4}
\end{equation} by the choice \eqref{choice of c0 2} of $c_0$, rather than the coarser inequality
\[N-1 > \frac{1}{(1-\delta)+\frac{1}{N}}\] in Section \ref{2k<=n proof}.

Finally, recall that the choice \eqref{choice of N 1} implies \eqref{choice of N 2}. Therefore, by solving \eqref{choice of N 1}, \eqref{choice of N 3}, \eqref{choice of N 4}, and incorporating the range \eqref{gamma} of $\gamma_0$, we obtain that,
\[0<\gamma_0<\frac{2k}{n}, \quad N_{0}<N \leq \frac{1}{1-\gamma_0},\] and
\[0<\delta<\min\left\{1-\frac{1}{N},\quad 1+ \frac{1}{N}-\frac{1}{(n-k+1)(N-1)}\right\},\] where
\[N_{0}:=\frac{1+\sqrt{1+4(n-k+1)^2}}{2(n-k+1)}<1+\frac{1}{n-k+1}.\]

For convenience, if we choose
\[\gamma_0=\frac{k}{n},\] then
\[N=\frac{1}{1-\gamma_0}=\frac{n}{n-k}>1+\frac{1}{n-k+1}\] would be a good choice satisfying all the requirements. Accordingly, a good choice for $\delta$ would be 
\[\delta=\frac{k}{2n}.\]

This proves the following refined version of Lemma \ref{the claim 2}, which we state together with Lemma \ref{the claim}.

\begin{lemma}\label{the refined claim}
When $2k \leq n$, there exists some choice of $N=N(n,k)>1$ such that if the test function $Q$ attains its maximum at some interior point $X_0$, then 
\[\kappa_{n}(X_0) \geq -C\] for some $C=C(n,k,\sigma,N)>0$ and if 
\[\kappa_1(X_0) \geq C(n,k,\sigma,N)\] is sufficiently large, then
\begin{align*}
&\ 2\sum_{i>m}\frac{F^{ii}h_{11i}^2}{\kappa_1(\kappa_1-\kappa_i)}-\sum_{i>m}\frac{F^{ii}h_{11i}^2}{\kappa_{1}^2}+N\sum_{i=1}^{n} \frac{F^{ii}(\nabla_i\nu^{n+1})^2}{(\nu^{n+1})^2}\\
&\ +(N-1)\sum_{i=1}^{n} F^{ii} -2N\sum_{i=1}^{n} F^{ii}\frac{u_i}{u}\frac{\nabla_i\nu^{n+1}}{\nu^{n+1}}\\
\geq &\ 0
\end{align*} at $X_0$.
\end{lemma}
\end{proof}

\begin{remark}
In our original proof in Section \ref{2k<=n proof}, we failed to observe the crucial, extra factor \eqref{the extra factor}.
\end{remark}
\begin{remark}
With Lemma \ref{the refined claim}, the curvature estimate \eqref{global C2 curvature} follows for all $k$ with $2k \leq n$ by the argument in Section \ref{the proof continued}. Combining this with the curvature estimate for $2k>n$ proved in Section \ref{2k>n proof}, Theorem \ref{the theorem} is now valid for all $k$.
\end{remark}

\begin{remark}
Incorporating the above refinement, the arguments in Section \ref{2k<=n proof} and Section \ref{2k>n proof} can be re-organized into one single proof that works for all $2 \leq k \leq n-1$. We nevertheless retain the original organization of the paper and present the refinement separately in this appendix.
\end{remark}
\end{document}